\documentclass[12pt]{article}
\usepackage{graphicx,amsmath,amsfonts,amssymb,amscd,bezier,amstext,makeidx}
\usepackage{fancyhdr}
\usepackage{amsthm}
\usepackage{amsmath}
\usepackage{bm}
\usepackage{color}
\usepackage{fancybox}
\usepackage{latexsym}
\usepackage{graphicx}
\usepackage{amsthm,amssymb,amscd}
\usepackage[all]{xy}
\usepackage{longtable}
\usepackage{comment}

\begin{document}

\newtheorem{definition}{Definition}[section]
\newtheorem{theorem}{Theorem}[section]
\newtheorem{corollary}{Corollary}[section]
\newtheorem{lemma}{Lemma}[section]
\newtheorem{proposition}{Proposition}[section]
\newtheorem{example}{Example}[section]
\newtheorem{remark}{Remark}[section]

\newfont{\bms}{msbm10 scaled 1200}
\def\matR{\mbox{\bms R}}
\def\matF{\mbox{\bms F}}
\def\matZ{\mbox{\bms Z}}
\def\matN{\mbox{\bms N}}
\def\matQ{\mbox{\bms Q}}
\def\matK{\mbox{\bms K}}
\def\matL{\mbox{\bms L}}
\def\matC{\mbox{\bms C}}
\def\matM{\mbox{\bms M}}
\newcommand{\B}{b\!\!\!b}
\newcommand{\A}{e\!\!\!a}

\title{Poincar\'e-Hopf Theorem for Isolated Determinantal Singularities}

\author{{ N. G. Grulha Jr.},   {M. S.  Pereira } and {H. Santana} }

\date{}

\maketitle

\begin{abstract}
\noindent Let $X$ be a projective algebraic $d$-variety endowed with isolated determinantal singularities, and let $\omega$ be a $1$-form on $X$ exhibiting a finite number of singularities (in the stratified sense). Under some technical conditions, we use two generalizations of Poincar\'e-Hopf index with the goal of proving a Poincar\'e-Hopf Type Theorem for $X$. 

\end{abstract}

\section*{Introduction}

The Poincaré-Hopf Theorem serves as a crucial link between combinatorial algebraic topology and differential topology, with the Euler characteristic of a manifold playing a pivotal role in establishing this connection. The Euler characteristic  is an essential and widely recognized invariant that has been present in mathematics since the early years of primary school, and its significance extends to encompass significant applications in theoretical physics.


In order to calculate the Euler characteristic for a smooth variety within the realm of differentiable structures, it is imperative to incorporate the Poincaré-Hopf index. However, when dealing with singular varieties, it becomes essential to extend the concept of the Poincaré-Hopf index to the singular case. In this context, several generalizations have been explored, as exemplified by the diverse approaches elucidated in \cite{BS, GSV, KT, MP}.

In \cite{BSS}, the authors furnish a demonstration of such a result specifically for the scenario where these isolated singularities manifest as complete intersections. Within this context, the existence and uniqueness of smoothing are established, thereby facilitating the formulation of a generalization for the Poincaré-Hopf index.

The subsequent step in advancing the research implies leveraging these new indices to establish proof of the Poincaré-Hopf Theorem for compact varieties characterized by isolated singularities of determinantal nature. This study specifically focuses on compact varieties exhibiting isolated determinantal singularities. To attain a version of the Poincaré-Hopf Theorem within this context, we employ techniques similar to those utilized in \cite{BSS}, while also incorporating new findings concerning determinantal singularities.

Let $X$ be a compact variety with isolated codimension $2$ determinantal singularities. In \cite{RP}, using the unicity of the smoothing, the authors define the Milnor number of $X$ as the middle Betti number of a generic fiber of the smoothing. In the more general setting of determinantal varieties, the results depend on the Euler characteristic of the stabilization given by the essential smoothing. In that paper, the authors also connect this invariant with the Ebeling and Gusein-Zade index of the $1$-form given by the differential of a generic linear projection defined on the variety.  

The cases that we consider in this paper are not covered by the ICIS setting. The non-standard behavior in our setting arises from the fact that we have both smoothable and non-smoothable singularities, and even in the smoothable case we must distinguish two separate situations:  unicity or not unicity of the smoothing. We consider two different Poincar\'e-Hopf index generalizations: one, denoted by $Ind_{PH},$ was considered by Ebeling and Gusein-Zade in \cite{EG} and can be seen as a generalization of the GSV-index \cite{GSV} and, the other, by $Ind_{PHN}$ defined by Ebeling and Gusein-Zade in \cite{EG}.


In Section 1, we provide a comprehensive exposition of the fundamental findings concerning determinantal varieties and the topology associated with essentially isolated determinantal singularities. Subsequently, in Section 2, we expound the indices of $1$-forms. Moreover, we establish theorems of the Poincaré-Hopf nature for projective surfaces with codimension $2$ singularities, projective $3$-varieties with codimension $2$ singularities, as well as for projective varieties that the singularities lack smoothability.

\begin{center}\textbf{Acknowledgements}\end{center}

The authors are grateful to professors Brasselet, Seade, and Ruas for their important suggestions about the theme of this paper. We also thank professors Ebeling and Zach for the fruitful conversations about their work, which is essential to this paper.

The first and second authors were supported by the São Paulo Research Foundation (FAPESP) , under grant $2019/21181-02.$ The first author would like to thank the Coordenação de Aperfeiçoamento de Pessoal de Nível Superior – Brasil (CAPES) for support through the MATH-AmSud program, Grant No. 88881.179491/2025-01. The second author was partially supported by Proex ICMC/USP on a visit to São Carlos, where part of this work was developed. The third author was partially supported by PROBAL (CAPES-DAAD), grant 88887.180502/2025-00. The authors also thank PROBAL (CAPES-DAAD), grant 88881.198862/2018-01 and FAPESP grant 2019/21181-0.

\section{Basic Definitions}

 Let \(M_{n,p}\) denote the set of all \(n \times p\) matrices with complex entries, and let \(M^{t}_{n,p} \subset M_{n,p}\) be the subset consisting of matrices of rank less than \(t\), where \(1 \leq t \leq \min(n,p)\). It is well known that \(M^{t}_{n,p}\) forms a singular algebraic variety of codimension \((n - t + 1)(p - t + 1)\), whose singular locus is precisely \(M^{t-1}_{n,p}\) (see \cite{Bruns}). This variety is usually referred to as the \emph{generic determinantal variety}.

\begin{definition}
Let \(F = (F_{ij}(x))\) be an \(n \times p\) matrix whose entries are complex analytic functions on an open set \(U \subset \mathbb{C}^r\) with \(0 \in U\), and let \(f\) be the function given by the collection of all \(t \times t\) minors of \(F\). We say that \(X_0\) is a \emph{determinantal variety} if it is defined by the equation \(f = 0\) and its codimension is \((n - t + 1)(p - t + 1)\).
\end{definition}

In this work, using the results of \cite{EG} and \cite{Mathias}, we present Poincaré--Hopf-type formulas for certain classes of compact varieties with isolated determinantal singularities. To apply the framework developed in \cite{Mathias}, we begin by recalling the notions of \emph{essentially nonsingular points} and \emph{essentially isolated determinantal singularities} (EIDS).
 
 \begin{definition}
 A point $x\in X = F^{-1}(M^t_{n,p} )$ is called essentially nonsingular if, at this point, the map $F$ is transversal to the corresponding stratum \mbox{$M^{i}_{n,p} \setminus M^{i-1}_{n,p}$} of the variety $M^t_{n,p}$, where $i = rkF(x) + 1$.  A germ $(X, 0) \subset(\mathbb{C}^r , 0)$ has an isolated essentially singular point at the origin if it has only essentially non-singular points in a punctured neighborhood of the origin in $X$.
\end{definition}

It is well known that complete intersections are always smoothable. In contrast, determinantal singularities exhibit a substantially more intricate deformation theory: they need not admit a smoothing, and when smoothings exist, they may fail to be unique. To work inside a deformation class that preserves the determinantal structure, Ebeling and Gusein–Zade introduced the notion of an \emph{essential smoothing}, obtained by a generic perturbation of the defining matrix which enforces transversality to the natural determinantal stratification.

\begin{definition}
Let \((X,0)\) be an essentially isolated determinantal singularity defined by a germ  \(F\colon (U,0)\to M_{n,p}\) of analytic matrix-valued functions, and let $X=F^{-1}(M_{n,p}^t)$.
An \emph{essential smoothing} of \(X\) is a subvariety \(\widetilde{X}\subset U\) defined by a perturbation  \(\widetilde{F}\colon U\to M_{n,p}\) of \(F\) such that \(\widetilde{F}\) is transversal to every stratum 
\(M^i_{n,p}\setminus M^{i-1}_{n,p}\) for all \(i\le t\).
\end{definition}
 
 Moreover, an essential smoothing becomes a genuine (i.e.\ smooth) smoothing only under additional numerical conditions. For example, for an EIDS of type \((m,n,t)\), the essential smoothing is smooth if and only if \(N < (m - t + 2)(n - t + 2),\) as shown by Ebeling and Gusein-Zade. In contrast, several important classes of determinantal varieties (most notably many codimension two determinantal varieties with isolated singularities) are in fact smoothable and admit a unique smoothing (\cite{EG}). This uniqueness allows for the definition of Milnor-type invariants, as illustrated by works such as \cite{RP} and \cite{BOT}, among others.

In order to apply \cite{Mathias} we need to introduce some notations.  Let $\bold{F}:\mathbb{C}^r\times\mathbb{C}\to (Mat(n,p;\mathbb{C}),0)\times (\mathbb{C},0)$ an stabilization of $F$. Note that $\bold{F}(x,u)=(F_u(x),u)$  with $F_0(x)=F(x)$ and $F_u$ is transversal to every stratum of $M^t_{n,p}\times\mathbb{C}$ for all $u\neq 0$ sufficiently small.  
Choose open sets $U\subset \mathbb{C}^r$, $W\subset\mathbb{C}$, $B\subset\mathbb{C}^r$ be a Milnor ball for $(X_0,0)$ em $U$ and a small representative $$F:U\times W\to Mat(n,p;\mathbb{C})\times W.$$
An essential smoothing could be seen as a family coming from a stabilization $X = \bold{F}^{-1}(M^t_{n,p}\times\mathbb{C})$ of the map $F$ (see \cite{Mathias}).  Let $\overline{X_u}\cong F_u^{-1}(M_{n,p}^t)\cap B$ the determinantal Milnor fiber.

Let $(X, 0)\subset (\mathbb{C}^r ,0)$ be the germ of a codimension $2$ determinantal variety with isolated singularity at the origin of dimension $d= 2,\, 3$. Pereira and Ruas in \cite{RP}
define the Milnor number of $X$ is by $\mu(X) = b_d(X)$, with $b_d(X)$ the $d$-th Betti number of $X$. Using the multiplicity $m_{d}(X)$ defined by Gaffney in \cite{Gaffney}, the following result can be obtained, which may be interpreted as a Lê–Greuel type formula for germs of Cohen–Macaulay determinantal varieties of codimension two with an isolated singularity at the origin.

\begin{theorem}(\cite{RP})\label{Miriam1.1}
\begin{itemize}
\item[a)] Let $(X,0)\subset(\mathbb{C}^4,0)$ be the germ of a determinantal surface with an isolated
singularity at the origin. Then,
$$m_2(X)=\mu(p^{-1}(0)\cap X)+\mu(X),$$
where $m_2(X)$ is the second polar multiplicity of  $X$.
\item[b)] Let $(X,0)\subset(\matC^5,0)$ be the germ of a determinantal variety of codimension $2$ with an isolated singularity at the origin. Then, $$m_3(X)=\mu(p^{-1}(0)\cap X)+\mu(X)+b_2(X_u),$$ where $b_2(X_t)$ is the $2$-th Betti number of the generic fiber of $X_u$ and $m_3(X)$ is the polar multiplicity of $X$.
\end{itemize}
\end{theorem}

In \cite{Mathias}, the complex link $L^{t,r}_{n,p}$ is introduced in the study of essentially isolated determinantal singularities. Since we only require the result below, we omit its definition and refer the reader to that reference for details. Combining the result from Section 4.1 with Corollary 3.6, it follows that, for a smoothable determinantal singularity with $2p>r$,
\[
L_{2,p}^{2,r}\cong_{ht} S^2,
\]
and therefore
\[
\overline{X}_u\cong_{ht} S^{2}\vee\bigvee_{i=1}^s S^{r-p+1}.
\]

In particular, if $X_0$ is a determinantal surface with isolated singularity in $\mathbb{C}^4,$ then $$\overline{X}_u\cong_{ht}S^{2}\vee\bigvee_{i=1}^{s} S^2\cong\bigvee_{i=1}^{s+1} S^2.$$ In this case, the Milnor number defined in \cite{RP} is the number of spheres appearing on the previous bouquet. 

If $X$ is a determinantal $3$-variety in $\mathbb{C}^5$,
$$\overline{X}_u\cong_{ht} S^2 \vee\bigvee_{i=1}^s S^3.$$

Therefore, $b_2(X_u)=1$ and using item $b)$ of the Theorem \ref{Miriam1.1}  we obtain the following consequence.

\begin{corollary}
Let $(X,0)\subset(\matC^5,0)$ be the germ of a determinantal variety of codimension $2$ with an isolated singularity at the origin. Then, $$m_3(X)=\mu(p^{-1}(0)\cap X)+\mu(X)+1.$$ 
\end{corollary}

\section{Index of $1$-Forms on Determinantal Varieties}

In this section, we recall the definitions of the indices of holomorphic $1$-forms on determinantal varieties with essentially isolated singularities, following the framework developed by Ebeling and Gusein-Zade (\cite{EG}). We focus on the case of isolated determinantal singularities, where the construction becomes simpler: the indices can be defined either via the zeros of the lifted form on an essential smoothing or through the behavior of the form on any of the canonical resolutions introduced in their work. We also summarize the basic relations among these indices that will be used later.

Let \( (X,0) \) be an isolated determinantal singularity given by a matrix \( F = (F_{ij}(x)) \), and let \( \widetilde{X} \) be an essential smoothing of \( X \). 
We now consider the Poincaré–Hopf–Nash index (PHN-index), defined as follows. 
Let \( \overline{X} \) be the total space of the Nash transform of the variety \( \widetilde{X} \), 
let \( \mathbb{T} \) be the Nash bundle over \( \widetilde{X} \), and 
let \( \Pi : \overline{X} \to \widetilde{X} \) be the associated projection. 
The \(1\)-form \( \omega \) induces a nonvanishing section \( \hat{\omega} \) 
of the dual bundle \( \hat{\mathbb{T}}^{*} \) over the preimage of the intersection 
\( \widetilde{X} \cap S_{\varepsilon} \), where \( S_{\varepsilon} \) is the sphere of radius 
\( \varepsilon \) centered at the origin. 
Note that \( \overline{X} \) is a smooth manifold, as proved in \cite{EG}, p.~06.

\begin{definition}(\cite{EG})\label{PHN}
The Poincar\'e–Hopf–Nash index of the $1$-form $\omega$ on the EIDS $(X,0)$, $ Ind_{PHN}(\omega, X, 0)=Ind_{PHN} \omega $, is the obstruction to extending the nonzero section $\hat{\omega}$  of the dual Nash bundle $\hat{\mathbb{T}}^*$  from the preimage of the boundary $S_{\varepsilon} = \partial B_{\varepsilon}$ of the ball $B_{\varepsilon}$ to the preimage of its interior, i.e., to the manifold $\overline{X}$ or, more precisely, its value (as an element of the cohomology group $H^{2d}(\Pi(\tilde{X}\cap B_{\varepsilon}),\Pi(\tilde{X}\cap S_{\varepsilon}))$ on the fundamental class of the pair \linebreak$(\Pi(\tilde{X}\cap B_{\varepsilon}),\Pi(\tilde{X}\cap S_{\varepsilon}))$.
\end{definition}

As remarkable in \cite{EG}, it follows from the property of the Euler obstruction that there is another possibility for the definition of the Poincar\'e–Hopf index.

\begin{proposition}
Let $(X,0)$ be an EIDS and $\omega$ a $1$-form on $(X,0)$. The Poincar\'e–Hopf index (PH-index), $Ind_{PH}\, \omega = Ind_{PH}(\omega, X, 0)$, of  $\omega$ on $(X,0)$ is the sum of the indices of the zeros of a generic perturbation $\tilde{\omega}$ of the 1-form $\omega$ on the essential smoothing $\tilde{X}$ appearing in the preimage of a neighborhood of the origin. 
 \end{proposition}

 If \(X\) is a smoothable singularity, then the definition given in \ref{PHN} coincides with that presented in \cite{BSS}, Section~3.4. Indeed, in the smooth case the Nash transform \( \overline{X} \) is isomorphic to \( \tilde{X} \).

We start considering smoothable isolated determinantal varieties, i.e., $r<(n-t+1)(p-t+1)$. In this context, the relation between the PHN-index and the radial index (as discussed in \cite{EG}, Section 3) is expressed by:
\begin{equation}\label{FM}
Ind_{PHN}(\omega,X,0)=Ind_{rad}(\omega,X,0)+(-1)^{\dim(X)}\overline{{\chi}}(\tilde{X}),
\end{equation}
with $\overline{{\chi}}(\tilde{X})={\chi}(\tilde{X})-1$. 

Let \( X \) be a compact algebraic variety with isolated singularities \( p_1, \dots, p_l \), where the germ \( (X_i,p_i) = (X, p_i) \), for \( i = 1, \dots, l \), represents a germ of an isolated singularity. Suppose \( \omega \) is a 1-form on \( X \) with isolated singularities. We then use the following consequence of the definition of the radial index:

\begin{equation}\label{Radial}
Ind_{rad}(\omega,X_i,p_i)=1+\sum_{j=1}^{s_i}Ind_{PHN}(\overline{\omega},X_i,q_j^i),
\end{equation}
where $\overline{\omega}$ is a $1$-form on $X_i\setminus B(p_i,\varepsilon_i')$ which coincides with $\omega$ on $X \cap \partial B(p_i,\varepsilon_i)$ and with a radial form on $X\cap \partial B(p_i,\varepsilon_i')$ and $q_j^i$ are the singularities of $\overline{\omega}$ on $X_i$ (see \cite{BSS}).

\section{Results}

The article \cite{BCO} addresses the classical question in algebraic geometry of what constraints are imposed on the singularities that can be afforded on a given
class of algebraic varieties. The central objective is to prove Theorem \ref{A}: any set of isolated $n$-dimensional algebraic singularities can be afforded on a simply connected projective variety. More precisely, we have:

\begin{theorem}\label{A} Let $(Y,y)$ be the germ of a given isolated singularity. There
exists a simply connected projective variety $X$ containing $Y$ and with $X\backslash\{y\}$
smooth.\end{theorem}


In \cite{SeadeSuwa, BSS}, the authors proved a Poincar\'e-Hopf type theorem for compact varieties with isolated complete intersection singularities. Using \cite{RP}, we present formulas that extend this result for the case of codimension $2$ projective varieties with smoothable and nonsmoothable isolated  determinantal singularities.

\begin{theorem}\label{2.1}
  Let $X \subset \mathbb{P}^r$ be a compact complex $d$-variety with isolated singularities $\{p_1, \dots, p_l\}$, such that each germ $(X_i, p_i) = (X, p_i)$ is a smoothable codimension 2 isolated determinantal singularity with a fixed representative matrix, all of the same type for $i=1,\ldots,l.$ Let $\omega$ be a 1-form on $X$ with a finite number of singularities $\{p_{l+1},\ldots,p_{l+m}\}$ (in the stratified sense) . Then
\[
 \sum_{i=1}^{l+m} Ind_{PHN}(\omega, X, p_i) = \chi(X) + \sum_{i=1}^l \left(1 + (-1)^d \overline{\chi}(\tilde{X}_i) \right).
\]                                         
\end{theorem}

\begin{proof}
 Let $0 < \varepsilon_i \ll 1$ and $X\cap B(p_i,\varepsilon_i)$ be a representative of the germ of smoothable isolated determinantal singularity $(X_i,p_i)$, $1 \leq i \leq l$, given by the $t \times t$ minors of a fixed representative $n \times p$ matrix. 
If the $1$-form $\omega$ does not take positive values on a fixed outward-pointing normal vector field along the boundary of $X'$, we choose $\varepsilon_i'<\varepsilon_i$ so that the open balls $B(p_i,\varepsilon_i')$ correspond to those used in the construction of the radial index, where
$$
X' = X \setminus \bigcup_{i=1}^l B(p_i,\varepsilon_i').
$$

Then we have the decomposition
\begin{eqnarray*}\label{Equação decomposição}
\sum_{i=1}^{l+m} \operatorname{Ind}_{PHN}(\omega,X,p_i)&=&
 \sum_{i=1}^l \operatorname{Ind}_{PHN}(\omega,X_i,p_i)
+
\sum_{j=1}^{m} \operatorname{Ind}_{PHN}(\omega,X',p_{l+j}),
\end{eqnarray*}

Note that $X'$ is a variety with boundary, thus
\[
\sum_{j=1}^m Ind_{PHN}(\omega, X', p_{l+j}) = \chi(X').
\]

On the other hand, since the Euler characteristic is additive, each $X\cap B(p_i,\varepsilon'_i)$ is contractible  and the boundary $\partial X_i$ (the link of the complex singularity) has zero Euler characteristic, we have 
\begin{eqnarray}\label{característica de Euler aditiva}
    \chi(X)=\chi(X')+\sum_{i=1}^{l}\chi(X\cap B(p_i,\varepsilon'_i))=\chi(X')+l.
\end{eqnarray}

Let $\tilde{X}$ be an essential smoothing of $X.$ Then $\tilde{X'}=X'.$
Let $\overline{\omega}$ be the $1$-form defined on $\tilde{X}$ such that it coincides with $\omega$ on $X'$ and, for each $i=1,\dots,l$, agrees with a radial $1$-form on $\partial X_i=X\cap \partial B(p_i,\varepsilon_i')$. 
Using Equations (\ref{FM}) and (\ref{Radial}), we have
\begin{align*}
\sum_{i=1}^l Ind_{PHN}(\omega,X_i,p_i)&=\sum_{i=1}^l\left(Ind_{rad}(\omega,X_i,p_i)+(-1)^d\overline{\chi}(\tilde{X}_i)\right)\\ 
                                                  & =\sum_{i=1}^l \left(1+\sum_{j=1}^{s_i}Ind_{PHN}(\overline{\omega},X_i,q_j^i)+(-1)^d\overline{\chi}(\tilde{X}_i)\right)\\
                                                  & =\sum_{i=1}^l \left(1+\chi(X_i)+(-1)^d\overline{\chi}(\tilde{X}_i)\right)\\
                                                   & =l + \sum_{i=1}^l \left(1+(-1)^d\overline{\chi}(\tilde{X}_i)\right),\\
\end{align*} 
where the points \(q_j^{\,i}\) denote the singularities of  \(\overline{\omega}\) on \(X_i\), and \(\tilde{X}_i\) is an essential smoothing of  \((X_i,p_i)\).

Hence,
\begin{align*}
\sum_{i=1}^{l+m} Ind_{PHN}(\omega,X,p_i)                                     &=l+\sum_{i=1}^l\left(1+(-1)^d\overline{\chi}(\tilde{X}_i)\right)+\chi(X')\\
& =\chi(X)+\sum_{i=1}^l\left(1+(-1)^d\overline{\chi}(\tilde{X}_i)\right).
\end{align*}
\end{proof}

\begin{corollary}\label{surfaces}
Let $X\subset \mathbb{P}^4$ be a projective algebraic surface with isolated determinantal singularities $p_1,\ldots,p_l$ such that each germ $(X_i, p_i) = (X, p_i)$, $i = 1, \dots, l$, is an isolated determinantal singularity of codimension 2 with fixed representative $2\times3$ matrices. Let $\omega$ be a 1-form on $X$ with a finite number of singularities $\{p_{l+1},\ldots,p_{l+m}\}$ (in the stratified sense) . Then
\[
\sum_{i=1}^{l+m} Ind_{PHN}(\omega, X, p_i) = \chi(X) + \sum_{i=1}^l 1+\mu(X_i).
\]
\noindent where $\mu(X_i)$ denotes the determinantal Milnor number of the germ $(X_i,p_i).$
\end{corollary}
\begin{proof}
    By Theorem \ref{2.1},
    \begin{eqnarray*}
         \sum_{i=1}^{l+m} Ind_{PHN}(\omega, X, p_i) &=& \chi(X) + \sum_{i=1}^l \left(1 + (-1)^2 \overline{\chi}(\tilde{X}_i) \right)\\
          &=&\chi(X) + \sum_{i=1}^l \chi(\tilde{X}_i)\\
         &=&\chi(X) + \sum_{i=1}^l 1+\mu(X_i).
    \end{eqnarray*}
\end{proof}

\begin{example}
Let $X \subset \mathbb{P}^4$ be the projective cone over the twisted cubic curve. We consider the homogeneous coordinates $[x_0 : x_1 : x_2 : x_3 : x_4]$ on $\mathbb{P}^4$. The surface $X$ is defined by the vanishing of the $2 \times 2$ minors of the $2 \times 3$ matrix of linear forms
\[
F = \begin{pmatrix}
x_0 & x_1 & x_2 \\
x_1 & x_2 & x_3
\end{pmatrix}.
\]
The singular locus of $X$ is given by $F^{-1}(M^1_{2,3})$, which corresponds to $x_0 = x_1 = x_2 = x_3 = 0$. Thus, $X$ has exactly one isolated determinantal singularity at the vertex $P_1 = [0:0:0:0:1]$. For this singularity, the determinantal Milnor number is $\mu(X_1) = 1$ see \cite{RP}.

To compute the Euler characteristic $\chi(X)$ using our result, we define a 1-form $\omega$ on $X$ induced by the $\mathbb{C}^*$-action on $\mathbb{P}^4$ with weights $wt(x_i) = i$ for $i = 0, \ldots, 4$. This action preserves the defining equations of $X$. The singularities of the induced 1-form $\omega$ on $X$ correspond to the fixed points of this action, which are the vertex $P_1 = [0:0:0:0:1]$ and the smooth points $P_2 = [1:0:0:0:0]$ and $P_3 = [0:0:0:1:0]$, and 

Locally at the smooth points $P_1$ and $P_2$, one has
\[
Ind_{PHN}(\omega, X, P_2) = 1 \quad \text{and} \quad Ind_{PHN}(\omega, X, P_3) = 1.
\]
At the singular point $P_1$, we may use OSCAR to compute 
\[
Ind_{PHN}(\omega, X, P_1) = 3.
\]
Summing the indices over all singularities of $\omega$ on $X$, we obtain
\[
\sum_{i=1}^3 Ind_{PHN}(\omega, X, P_i) = 1 + 1 + 3 = 5.
\]
Applying Corollary \ref{surfaces}, we have
\[
\sum_{i=1}^3 Ind_{PHN}(\omega, X, P_i) = \chi(X) + 1 + \mu(X_0).
\]
Hence, $ \chi(X) = 3.$

\end{example}

\begin{remark}
For simple space curves investigated by Fr\"uhbis-Kr\"uger in \cite{FK} we can prove an analogous result using the Milnor number of an arbitrarily reduced curve singularity defined by Buchweitz and Greuel in \cite{Bg}.
\end{remark}

\begin{theorem}\label{teo42}
Let \( X \subset \mathbb{P}^5 \) be a projective algebraic \(3\)-variety with isolated singularities  \( p_{1}, \ldots, p_{l} \), such that each germ \( (X,p_{i}) \), \( i = 1, \ldots, l \), 
is a codimension \(2\) determinantal \(3\)-variety with an isolated singularity.  Let $\omega$ be a 1-form on $X$ with a finite number of singularities $\{p_{l+1},\ldots,p_{l+m}\}$ (in the stratified sense) . Then

\begin{equation}
\sum_{i=1}^{l+m} Ind_{PHN} (\omega, X,p_i)=\chi(X)+\sum_{i=1}^l \mu (X_i).
\end{equation}
\end{theorem}

\begin{proof}

By Theorem \ref{2.1}, we have

\begin{align*}
 \sum_{i=1}^{3} Ind_{PHN}(\omega, X, p_i) &= \chi(X) + \sum_{i=1}^l \left(1 + (-1)^3 \overline{\chi}(\tilde{X}_i) \right)\\
 &=  \chi(X) + l - \sum_{i=1}^l \overline{\chi}(\tilde{X}_i)\\
 &=  \chi(X)+ l +\sum_{i=1}^l \left( {\mu} (X_i)-1 \right)\\
 &= \chi(X)+ \sum_{i=1}^l {\mu} (X_i)
\end{align*}
since $ b_2(X_i)=1, i=1,\ldots,l$. 
\end{proof}

Let $(X,0)=F^{-1}(M_{n,p}^t)\subset\mathbb{P}^r$ be an isolated nonsmoothable determinantal singularity, i.e, $r = (n-t+2)(p-t+2)$ and $\tilde{X}$ be a essential smoothing of $X$. In this case, the relation between the PHN-index and the radial index present in \cite{EG1} is
\begin{align}\label{tenth}
Ind_{PHN} &(\omega, X, 0)=\nonumber \\
                                          &Ind_{rad} (\omega, X, 0) + (-1)^{\dim X} \overline{\chi}(\tilde{X},0)+ (-1)^{n+p+1}(p - t + 1)\chi(\tilde{X}^{t-1})
\end{align}
with $X^{t-1}=F^{-1}(M_{n,p}^{t-1})$. In the next result, we consider a non-smoothable determinantal singularity with an isolated singularity.

\begin{theorem}\label{2.4}
 Let $X \subset \mathbb{P}^r$ be a compact complex $d$-variety with isolated singularities $\{p_1, \dots, p_l\}$, such that each germ $(X_i, p_i) = (X, p_i)$ is a nonsmoothable codimension 2 isolated determinantal singularity with fixed representative matrix, all of the same type for $i=1,\ldots,l.$  Let $\omega$ be a 1-form on $X$ with a finite number of singularities $\{p_{l+1},\ldots,p_{l+m}\}$ (in the stratified sense) . Then

\begin{align*}
    \sum_{i=1}^{l+m} Ind_{PHN}&(\omega,X,p_i)=\\&\chi(X)+\sum_{i=1}^l \left(1+(-1)^{d}\overline{\chi}(\tilde{X}_i)+(-1)^{n+p+1}(p-t+1)\chi(X_i^{t-1})\right).
\end{align*}

\end{theorem}
\begin{proof}
    
 Let $0 < \varepsilon_i \ll 1$ and $X\cap B(p_i,\varepsilon_i)$ be a representative of the germ of smoothable isolated determinantal singularity $(X_i,p_i)$, $1 \leq i \leq l$, given by the $t \times t$ minors of a fixed representative $n \times p$ matrix. 
If the $1$-form $\omega$ does not take positive values on a fixed outward-pointing normal vector field along the boundary of $X'$, we choose $\varepsilon_i'<\varepsilon_i$ so that the open balls $B(p_i,\varepsilon_i')$ correspond to those used in the construction of the radial index, where
$$
X' = X \setminus \bigcup_{i=1}^l B(p_i,\varepsilon_i')
$$

Then we have the decomposition
\begin{eqnarray*}\label{Equação decomposição}
\sum_{i=1}^{l+m} \operatorname{Ind}_{PHN}(\omega,X,p_i)&=&
 \sum_{i=1}^l \operatorname{Ind}_{PHN}(\omega,X_i,p_i)
+
\sum_{j=1}^{m} \operatorname{Ind}_{PHN}(\omega,X',p_{l+j}).
\end{eqnarray*}

Since $X'$ is a variety with boundary,
\[
\sum_{j=1}^m Ind_{PHN}(\omega, X', p_{l+j}) = \chi(X').
\]

Using relation (\ref{tenth}), we have
\begin{align*}
\sum_{i=1}^{l} &\operatorname{Ind}_{PHN}(\omega,X,p_i)\\
&= \sum_{i=1}^l \Big( \operatorname{Ind}_{rad}(\omega,X_i,p_i)+(-1)^d\overline{\chi}(\tilde{X}_i)+(-1)^{n+p+1}(p-t+1)\,\chi(\tilde{X}_i^{\,t-1})\Big) \\
&= \sum_{i=1}^l\Big(1+\sum_{j=1}^{s_i}\operatorname{Ind}_{PHN}(\overline{\omega},X_i,q_j^i)+(-1)^d\overline{\chi}(\tilde{X}_i)+(-1)^{n+p+1}(p-t+1)\,\chi(\tilde{X}_i^{\,t-1})\Big)\\
&= l+ \sum_{i=1}^l\Big(\chi(X_i)+(-1)^d\overline{\chi}(\tilde{X}_i)+(-1)^{n+p+1}(p-t+1)\,\chi(\tilde{X}_i^{\,t-1})\Big)\\
&= l+ \sum_{i=1}^l\Big(1+(-1)^d\overline{\chi}(\tilde{X}_i)+(-1)^{n+p+1}(p-t+1)\,\chi(\tilde{X}_i^{\,t-1})\Big)\\
\end{align*}

Hence, 

\begin{align*}
    \sum_{i=1}^{l+m} &\operatorname{Ind}_{PHN}(\omega,X,p_i)\\&=l+ \sum_{i=1}^l\Big(1+(-1)^d\overline{\chi}(\tilde{X}_i)+(-1)^{n+p+1}(p-t+1)\,\chi(\tilde{X}_i^{\,t-1})\Big)+\chi(X')\\
    &=\chi(X)+\sum_{i=1}^l\Big(1+(-1)^d\overline{\chi}(\tilde{X}_i)+(-1)^{n+p+1}(p-t+1)\,\chi(\tilde{X}_i^{\,t-1})\Big).
\end{align*}
\end{proof}

\begin{corollary}
On the conditions of above theorem if $(X,p_i)$ is Cohen Macaulay of codimension $2$, then
$$\sum_{i=1}^{l+m} Ind_{PHN}(\omega,X,p_i)=\chi(X)+\sum_{i=1}^l \left(1+(-1)^{d}\overline{\chi}(\tilde{X}_i)+2\chi(X_i^{n-1})\right).$$
\end{corollary}

\begin{corollary}
 Let $X \subset \mathbb{P}^6$ be a compact complex $d$-variety with isolated singularities $\{p_1, \dots, p_l\}$, such that each germ $(X_i, p_i) = (X, p_i)$ is a nonsmoothable codimension 2 isolated determinantal singularity with fixed representative matrix, all of the same type for $i=1,\ldots,l.$  Let $\omega$ be a 1-form on $X$ with a finite number of singularities $\{p_{l+1},\ldots,p_{l+m}\}$ (in the stratified sense) . Then
\begin{equation}
\sum_{i=1}^{l+m} Ind_{PHN}(\omega,X,p_i)=\chi(X)+\sum_{i=1}^l \left(\chi(\tilde{X}_i)+2\chi(X_i^{1})\right)= \chi(X)+\sum_{i=1}^l \chi(\tilde{X}_i)+2l.\nonumber
\end{equation}
\end{corollary}

%
%
%

\end{document}